\documentclass[review]{elsarticle}

\usepackage{hyperref}
\usepackage{amsmath,amsthm,amssymb}
\usepackage{graphicx}
\usepackage{enumitem}
\usepackage{mathtools}

\DeclarePairedDelimiter\floor{\lfloor}{\rfloor}
\journal{Journal of Mathematical Analysis and Applications}

\begin{document}
\newcommand{\R}{\mathbb{R}}

\newtheorem{Def}{Definition}[section]
\newtheorem{claim}{Claim}[section]
\newtheorem{lemma}{Lemma}[section]
\newtheorem{prop}{Proposition}[section]
\newtheorem{conj}{Conjecture}[section]
\newtheorem{theorem}{Theorem}[section]
\newtheorem{coro}{Corollary}[section]

\begin{frontmatter}

\title{A constructive proof of the Bollob\'as-Varopoulos theorem}

\author{Dylanger S. Pittman}
\address{400 Dowman Drive, Atlanta}

\ead{dpittm2@emory.edu}

\begin{abstract}
The Bollob\'as-Varopoulos theorem is an analogue of Hall's matching theorem on non-atomic measure spaces. Here we prove a finite version with a completely constructive proof. 
\end{abstract}

\begin{keyword}
combinatorics\sep measure theory \sep  discrete math \sep graph theory
\MSC[2010] 00-01\sep  99-00
\end{keyword}

\end{frontmatter}

\section{Introduction}
In 1972, Bella Bollob\'as and Nicholas Varopoulos \cite{Boll} proved an analogue to Hall's matching theorem using a short clever argument that involves characteristic function methods similar to those introduced in \cite{Vara}. Here, we reprove a finite version of the same theorem, with a fully constructive method. Here is our result.

\begin{theorem}\label{conthall}
Let $(\Omega,S,\nu)$ be a finite non-atomic measure and $A$ a measurable set in $S$ with $\nu (A)>0$ and subsets $A_1,A_2,\dots, A_n$ and let $m_1,m_2,\dots,m_n >0$.

Then there are disjoint subsets $B_k\subseteq A_k$ with $\nu(B_k)=m_k$ for all $k\in [n]$ if and only if 

\[\nu\left(\bigcup_{i\in I} A_i\right)\ge \sum_{i\in I} m_i\]

for all $I\subseteq[n].$
\end{theorem}

\section{Background and lemmas}

\subsection{Notation and non-atomic measures}

We use the following standard notation for the rest of the write up: $2^Q$ is the power set of a set $Q$, $|Q|$ denotes the cardinality of a set $Q$, $[n]=\{1,2,\dots, n\}$, and $\mathbb{N}=\{1,2,\dots\}$ are the natural numbers.

It would be beneficial to use to restate some of the key theorems and definitions that we use. 

Let $(\Omega, S, \nu)$ be a $\sigma$-finite measure. Then an {\it atom} of $\nu$ is a set $A \in S$ with $\nu(A) > 0$ such that for all $C \in S$ with $C \subset A$, either $\nu(C) = 0$ or $\nu(C) = \nu(A)$. By $\sigma$-finiteness, we have $\nu(A) <  \infty$. $(\Omega, S, \nu)$ or $\nu$ is called {\it non-atomic} if it has no atoms. Equivalently, $(\Omega, S, \nu)$ or $\nu$ is called {\it non-atomic} if for any measurable set $A$ with $\nu (A)>0$ there exists a measurable subset $B$ of $A$ such that $\nu (A)>\nu (B)>0.$

The following corollary is a consequence of Proposition A.1 in \cite{RMR}. This is coincidentally the theorem proved first by  Wacław Sierpiński, see \cite{Ser}.

\begin{coro}\label{mcoro1}
Let $(\Omega, S, \nu)$ be a non-atomic finite measure with $\nu(\Omega) > 0$. Then if $A$ is a measurable set in $S$ with $\nu (A)>0$, then for any real number $c$ with $\nu (A)\geq c\geq 0$ there exists a measurable subset B of A such that $\nu(B) = c.$
\end{coro}

We also have another corollary following from Proposition A.2 in \cite{RMR}.

\begin{coro}\label{mcoro2}
Let $(\Omega, S, \nu)$ be a finite non-atomic measure and $A$ a measurable set in $S$ with $\nu (A)>0$. Let $r_i$ for $i = 1,\dots, n$ be numbers with $r_i > 0$ and $\sum_{i=1}^n r_i = \nu(A).$ Then $A$ can be decomposed as a union of disjoint sets $R_i \in S$ with $\nu(R_i) = r_i$ for $i = 1, \dots, n$.
\end{coro}

\subsection{Generalizations of Hall's matching theorem}

In addition to Corollaries \ref{mcoro1} and \ref{mcoro2}, we also need to use Hall's matching theorem, see (\cite{Hall}, theorem 2.1.2 in \cite{tbook}). Hall's theorem itself will not be covered in this write-up as there is a lot of literature already in existence, not to mention it would require me writing a few pages of elementary graph theory to cover the requisite background to make the statement of the theorem understandable.

The following corollary is a generalization of Hall's matching theorem. It is in fact Exercise 2.9 on p.54 in \cite{tbook} and it is proven with a clever application of Hall's matching theorem.

\begin{coro}\label{Gcoro1}
Let $A$ be a finite set with subsets $A_1,\dots, A_n$, and let $d_1,\dots,d_n\in \mathbb{N}$. Then there are disjoint subsets $D_k\subseteq A_k$, with $|D_k|=d_k$ for all $k\in [n]$, if and only if 

\[\left|\bigcup_{i\in I} A_i\right| \ge \sum_{i\in I} d_i\]

for all $I\subseteq [n]$.
\end{coro}

The above claim is a generalization of Hall's matching theorem because if we set $d_i=1$ for all $i\in[n]$, then it would be an equivalent statement. If we consider the case where each point has the same weight of $\xi>0$, then we obtain the following corollary. 

\begin{coro}\label{Gcoro2}
Let $A$ be a finite set with subsets $A_1,\dots, A_n$ and let $d_1,\dots,d_n\in \mathbb{N}$. Given a real $\xi>0,$ let $(A,2^A,\eta)$ be a discrete measure defined as $\eta(X)=\xi|X|$ for all $X\in 2^A$.

Then there are disjoint subsets $D_k\subseteq A_k$, with $\eta(D_k)=\xi d_k$ for all $k\in [n]$, if and only if 

\[\eta\left(\bigcup_{i\in I} A_i\right) =\xi\left|\bigcup_{i\in I} A_i\right|\ge \xi\sum_{i\in I} d_i\]

for all $I\subseteq [n]$.
\end{coro}

If $\xi=1$, then will obtain Corollary \ref{Gcoro1}. So, the above Corollary \ref{Gcoro2} is a generalization of Corollary \ref{Gcoro1}.

If we replace weighted discrete points with disjoint subsets with the same measure and apply that to Corollary \ref{Gcoro2}, we obtain the next corollary.

\begin{coro}\label{Gcoro3}
Given a real $\xi>0$, let $(\Omega, S, \nu)$ be a finite measure and $A=\{A_1,A_2,\dots, A_\ell\}$ be a finite collection of disjoint measurable sets in $S$ where $\nu(A_i)=\xi$. Also let $d_1,\dots,d_n\in \mathbb{N}$ and $\alpha_1,\dots,\alpha_n\subseteq A$. 

Then there are disjoint subsets $D_k\subseteq \alpha_k$, with $\nu\left(\bigcup_{D\in D_k}D\right)=\xi d_k$ for all $k\in [n]$, if and only if 

\[\nu\left(\bigcup_{i\in I} \left(\bigcup_{B\in \alpha_i}B\right)\right) =\xi\left|\bigcup_{i\in I} \alpha_i\right|\ge \xi\sum_{i\in I} d_i\]

for all $I\subseteq [n]$.
\end{coro}

\begin{proof}
Let $(A,2^A,\eta)$ be a discrete measure defined as $\eta(X)=\xi|X|$ for all $X\in 2^A.$ Then by Corollary \ref{Gcoro2}, there are disjoint subsets $D_k \subseteq \alpha_k$, with $\eta(D_k)=\xi d_k$ for all $k\in [n]$, if and only if 
\[\eta\left(\bigcup_{i\in I}\alpha_i\right)=\xi \left|\bigcup_{i\in I}\alpha_i\right|\ge \xi\sum_{i\in I} d_i\]
for all $I\subseteq [n]$. 

Since $\bigcup_{i\in I}\alpha_i$ is a finite collection of disjoint measurable sets, then 
\begin{align*}
    \eta\left(\bigcup_{i\in I}\alpha_i\right)&=\eta\left(\bigcup_{B\in\bigcup_{i\in I}\alpha_i}\{B\}\right)\\
    &=\sum_{B\in\bigcup_{i\in I}\alpha_i}\eta(\{B\})\\
    &=\sum_{B\in\bigcup_{i\in I}\alpha_i}\nu(B)\\
    &=\nu\left(\bigcup_{B\in\bigcup_{i\in I}\alpha_i} B\right)\\
    &=\nu\left(\bigcup_{i\in I} \left(\bigcup_{B\in \alpha_i}B\right)\right).
\end{align*}

\end{proof}

It also good to mention that Corollary \ref{Gcoro3} is a generalization of Corollary \ref{Gcoro2}. As \ref{Gcoro3} deals with disjoint sets with measure, while \ref{Gcoro2} exclusively deals with the case where each set is a single point.

\section{Main Result}

Now to prove Theorem \ref{conthall}.

\begin{proof}

If such $B_k$'s exist, then for all $I$, $\nu\left(\bigcup_{i\in I} A_i\right)\ge \sum_{i\in I} m_i.$ 

Conversely, let $Q$ be a nonempty subset of $[n]$, then we can define

\[S_Q=\left(\bigcap_{i\in Q} A_i \right)\setminus \left(\bigcap_{i\not\in  Q} A_i\right).\]

Note that if $P$ is also a nonempty subset of $[n]$ and $Q\neq P$, then $S_Q\cap S_P=\varnothing$.
It would also be helpful to note that 

\[A_k=\bigcup_{Q\in\{ P\in 2^{[n]}| k\in P\}} S_Q\]

and 

\[\bigcup_{i\in [n]} A_i=\bigcup_{Q\in 2^{[n]}\setminus \{\varnothing\}} S_Q.\]

Assume $\xi>0.$ Then by Corollary \ref{mcoro1}, for each nonempty $Q\subseteq[n]$ there exists a subset of $S_Q$ of measure $\xi\floor*{\frac{\nu(S_Q)}{\xi}};$ let this subset be denoted as $E_{Q,\xi}$. Then by Corollary \ref{mcoro2}, $E_{Q,\xi}$ can be partitioned into $\floor*{\frac{\nu(S_Q)}{\xi}}$ subsets of measure $\xi$. We can now define a set

\[A_{k,\xi}=\bigcup_{Q\in\{ P\in 2^{[n]}| k\in P\}} E_{Q,\xi}.\]

Now consider the number $\floor*{\frac{m_k}{\xi}}$, we are given $m_k$ by the theorem statement, which by Corollaries \ref{mcoro1} and \ref{mcoro2} is the maximum number of disjoint subsets of measure $\xi$ for a set of measure $m_k$ and define $d_{k,\xi}=\floor*{\frac{m_k}{\xi}}-2^{n+1}$. We want $d_{k,\xi}$ to be positive integers, so note that if $0<\xi\le\frac{m_k}{2^{n+1}+1}$, then $d_{k,\xi}>0$.

Observe that given a nonempty $Q\subseteq [n]$ we have that
\begin{align*}
    \left|\nu\left(\bigcup_{i\in Q}A_i\right)-\nu\left(\bigcup_{i\in Q}A_{i,\xi}\right)\right| &=\left|\nu\left(\bigcup_{P\in \{T\in 2^{[n]}| T\cap Q\neq \varnothing\}} S_P\right)-\nu\left(\bigcup_{P\in\{T\in 2^{[n]}| T\cap Q\neq \varnothing\}} E_{P,\xi}\right)\right|\\
    &=\left|\sum_{P\in\{T\in 2^{[n]}| T\cap Q\neq \varnothing\}} \nu(S_P) -\sum_{P\in\{T\in 2^{[n]}| T\cap Q\neq \varnothing\}} \nu\left( E_{P,\xi}\right)\right|\\
    &= \left|\sum_{P\in\{T\in 2^{[n]}| T\cap Q\neq \varnothing\}} \left[\nu(S_P)- \nu\left( E_{P,\xi}\right)\right] \right|\\
    &<\left|\sum_{P\in\{T\in 2^{[n]}| T\cap Q\neq \varnothing\}} \xi \right|=\xi (2^n-2^{n-|Q|}).
\end{align*}

Thus as $\xi \to 0$, $\nu\left(\bigcup_{i\in Q}A_{i,\xi}\right)\to \nu\left(\bigcup_{i\in Q}A_i\right)$.

Also note that 

\[\nu\left(\bigcup_{i\in Q}A_i\right)\ge\nu\left(\bigcup_{i\in Q}A_{i,\xi}\right).\]

We use $d_{k,\xi}$ because if we are given a nonempty $Q\subseteq [n]$ then by the work we did previously, $\nu\left(\bigcup_{i\in Q} A_{i,\xi}\right)> \nu\left(\bigcup_{i\in Q} A_{i}\right)-\xi(2^n-2^{n-|Q|})$. 

Since the theorem hypothesis states that $\nu\left(\bigcup_{i\in Q} A_i\right)\ge \sum_{i\in Q} m_i$, then 

\begin{align*}
    \nu\left(\bigcup_{i\in Q} A_i\right)-\xi(2^n-2^{n-|Q|})&\ge \left(\sum_{i\in Q} m_i\right)-\xi(2^n-2^{n-|Q|})\\
    &\ge \xi\left(\sum_{i\in Q} \floor*{\frac{m_i}{\xi}}\right)-\xi(2^n-2^{n-|Q|})\\
    &\ge\xi\left(\sum_{i\in Q} \floor*{\frac{m_i}{\xi}}\right)-2^{n+1}\xi\\
    &\ge\xi\sum_{i\in Q} \left(\floor*{\frac{m_i}{\xi}}-2^{n+1}\right)=\xi\sum_{i\in Q} d_{i, \xi}.
\end{align*}

In summary, for sufficently small $\xi$,

\[\nu\left(\bigcup_{i\in Q} A_{i,\xi}\right)>\nu\left(\bigcup_{i\in Q} A_i\right)-\xi(2^n-2^{n-|Q|})\ge \xi\sum_{i\in Q} d_{i, \xi}>0. \]

 Therefore for all $I\subseteq [n]$ and a sufficiently small $\xi$ we have

\[\nu\left(\bigcup_{i\in I}A_{i,\xi}\right)> \xi \sum_{i\in I} d_{i,\xi}>0.\]

 Therefore by Corollary \ref{Gcoro3}, for sufficiently small $\xi$ there exists disjoint subsets $B_{k,\xi}\subseteq A_{k,\xi}$ such that $\nu(B_{k,\xi})=\xi d_{k,\xi}$ for all $k\in[n].$

 Given a nonempty $Q\subseteq [n]$ we have that $\sum_{i\in Q}m_i\ge\xi\sum_{i\in Q} d_{i,\xi}$. Therefore

\begin{align*}
    \left|\sum_{i\in Q}m_i-\xi\sum_{i\in Q} d_{i,\xi}\right| &=\left|\sum_{i\in Q}(m_i- \xi d_{i,\xi})\right| \\
    &=\left|\sum_{i\in Q}\left(m_i- \xi\floor*{\frac{m_i}{\xi}}+\xi2^{n+1} \right)\right|\\
    &=2^{n+1}\xi|Q|+\left|\sum_{i\in Q}\left(m_i- \xi\floor*{\frac{m_i}{\xi}} \right)\right|.
\end{align*}

Keep in mind that $ \lim_{\xi \to 0} m_i-|\xi|\le \lim_{\xi \to 0} \xi\floor*{\frac{m_i}{\xi}}\le\lim_{\xi \to 0} m_i+|\xi|.$ Therefore $\lim_{\xi \to 0} \xi\floor*{\frac{m_i}{\xi}}=m_i.$

Thus as $\xi\to 0$ we have that $\nu(B_{k,\xi})=\xi d_{k,\xi}\to m_k$ for all $k\in [n]$ and

\[\nu\left(\bigcup_{i\in I}A_{i,\xi}\right)> \xi \sum_{i\in I} d_{i,\xi}\longrightarrow \nu\left(\bigcup_{i\in I} A_i\right)\ge \sum_{i\in I} m_i\]

for all $I\subseteq [n].$
 
Given a sufficiently  small $\xi>0$, let us define $\xi_i=\frac{\xi}{2^i}$ for $i\in \mathbb{N} \cup\{0\}.$ So for each $k\in [n]$ we can now create a sequence $\{\nu(B_{k,\xi_i})\}_{i=0}^\infty.$ Since this sequence is monotonically increasing, by our previous work and Corollary \ref{mcoro1} for every $k\in [n]$ there exists a sequence of sets 

\[B_{k,\xi_0}\subseteq B_{k,\xi_1}\subseteq B_{k,\xi_2}\subseteq \cdots\] 

such that for every $i\in \mathbb{N} \cup\{0\}$: $B_{k,\xi_i} \subseteq A_k$, $\nu(B_{k,\xi_i})=\xi_id_{k,\xi_i}$, and $B_{k,\xi_i} \cap B_{k',\xi_i}=\varnothing$ when $k\neq k'.$
Let us define 

\[B_{k,0}=\bigcup_{i=0}^\infty B_{k,\xi_i}.\]

By our previous work, we know that $m_k=\nu\left(B_{k,0}\right).$ We now need to prove two more things: 

\begin{enumerate}
    \item there exists a set $B_k \subseteq B_{k,0}$ such that $B_k\subseteq A_k$ and $\nu(B_k)=\nu(B_{k,0})$
    \item there exists sets $B_{k_1}\subseteq B_{k_1,0}$ and $B_{k_2}\subseteq B_{k_2,0}$ where $B_{k_1}\cap B_{k_2}=\varnothing$, $\nu(B_{k_1})=\nu(B_{k_1,0})$ and $\nu(B_{k_2})=\nu(B_{k_2,0})$ when $k_1\neq k_2.$
\end{enumerate}

If $B_{k,0}$ is not a subset of $A_k$ then $B_{k,0}\cap A_k^c \neq \varnothing.$ If $\nu(B_{k,0}\cap A_k^c)>0$, then there exits an $\alpha \in \mathbb{N} \cup\{0\}$ such that $\nu(B_{k,\xi_\alpha}\cap A_k^c)>0$, which contradicts $B_{k,\xi_\alpha}\subseteq A_k.$ Therefore $\nu(B_{k,0}\cap A_k^c)=0$. Thus there exists a set of measure zero $N_k$ such that $\nu(B_{k,0}\setminus N_k)=m_k$ and $B_{k,0}\setminus N_k \subseteq A_k$.

Similarly if we are given distinct $k_1,k_2 \in [n]$ and $B_{k_1,0}\cap B_{k_2,0} \neq \varnothing$, then if $\nu(B_{k_1,0}\cap B_{k_2,0})>0$, then there exits an $\alpha \in \mathbb{N} \cup\{0\}$ such that $\nu(B_{k_1,\xi_\alpha}\cap B_{k_2,\xi_\alpha})>0$. This contradicts $B_{k_1,\xi_\alpha}\cap B_{k_2,\xi_\alpha}=\varnothing,$ thus $\nu(B_{k_1,0}\cap B_{k_2,0})=0.$ Thus there exists sets of measure zero $E_{k_1}$ and $E_{k_2}$ such that $(B_{k,0}\setminus E_{k_1})\cap(B_{k',0}\setminus E_{k_2})= \varnothing$ if $k_1\neq k_2.$

So there exists disjoint sets $B_k \subseteq A_k$ such that $\nu(B_k)=m_k$ for all $k\in [n]$. Our proof is complete.

\end{proof}

\section*{Acknowledgements}

The author conjectured the main result before knowing that it had a connection with graph theory. He would like to thank Professor Hao Huang for pointing out its combinatorial connection and Professor Shanshuang Yang for his proofreading and verification.

\bibliography{mybibfile}

\end{document}